\documentclass[12pt]{amsart}
\usepackage{lmodern}

\usepackage[T1]{fontenc}
\usepackage[latin9]{inputenc}
\usepackage{geometry}
\usepackage{amsthm}
\usepackage{amsbsy}
\usepackage{amstext}
\usepackage{amssymb}
\usepackage{esint}

\makeatletter

\newcommand{\noun}[1]{\textsc{#1}}

\numberwithin{equation}{section}
\numberwithin{figure}{section}
\theoremstyle{plain}
\newtheorem{thm}{\protect\theoremname}
  \theoremstyle{plain}
  \newtheorem{prop}[thm]{\protect\propositionname}
  \theoremstyle{plain}
  \newtheorem{lem}[thm]{\protect\lemmaname}
  \theoremstyle{plain}
  \newtheorem{cor}[thm]{\protect\corollaryname}

\usepackage{mathrsfs}
\usepackage{wrapfig}
\usepackage[all]{xy}
\usepackage{bbm}
\newcommand{\one}{\mathbbm{1}}
\newcommand{\ti}[1]{\tilde {#1}}
\newcommand{\eps}{\varepsilon}
\newcommand{\defeq}{\stackrel{\rm{def}}{=}}
\newcommand{\cM}{{\mathcal M}}

\newcommand{\cO}{{\mathcal O}}

\newcommand{\scC}{\mathscr C}
\newcommand{\scL}{\mathscr L}
\newcommand{\scP}{\mathscr P}

\newcommand{\bU}{\boldsymbol U}
\newcommand{\bV}{\boldsymbol V}

\newcommand{\bga}{\boldsymbol{\gamma}}
\newcommand{\bbe}{\boldsymbol{\beta}}

\newcommand{\N}{{\mathbb N}}
\newcommand{\R}{{\mathbb R}}

\newcommand{\supp}{\operatorname {supp}}
\newcommand{\e}{\operatorname{e}}

\renewcommand{\i}{\operatorname{i}}
\renewcommand{\top}{\operatorname{top}}

\newcommand{\inj}{\operatorname{inj}}
  \providecommand{\corollaryname}{Corollary}
  \providecommand{\lemmaname}{Lemma}
  \providecommand{\propositionname}{Proposition}
\providecommand{\theoremname}{Theorem}
\usepackage{amssymb}

\makeatother

  \providecommand{\corollaryname}{Corollary}
  \providecommand{\lemmaname}{Lemma}
  \providecommand{\propositionname}{Proposition}
\providecommand{\theoremname}{Theorem}

\begin{document}

\title{exponential gaps in the length spectrum}

\author{Emmanuel Schenck}

\address{Laboratoire d'analyse, géométrie et applications, Université Paris
13, CNRS UMR 7539, 93430 Villetaneuse, France.}

\email{schenck@math.univ-paris13.fr }
\begin{abstract}
We present a separation property for the gaps in the length spectrum
of a compact Riemannian manifold with negative curvature. In arbitrary
small neighborhoods of the metric for some suitable topology, we show
that there are negatively curved metrics with a length spectrum exponentially
separated from below. This property was previously known to be false
generically. 
\end{abstract}
\maketitle

\section{Introduction}

The Anosov property of the geodesic flow on a compact manifold with
negative curvature provides a great abundance of closed geodesics
: the counting function satisfies 
\begin{equation}
\sharp\{\gamma:\ \ell(\gamma)\leq T\}\sim\frac{\e^{h_{\top}T}}{2h_{\top}T},\quad T\to+\infty,\label{eq: Margulis}
\end{equation}
where $h_{\top}>0$ denotes the topological entropy of the flow and
$\ell(\gamma)$ the length of a closed, unoriented geodesic $\gamma$.
We refer the reader to \cite{Mar69,MarSha04_anosov,PaPo90} for comprehensive
studies of periodic orbits for hyperbolic systems.

For generic metrics in negative curvature, distinct closed geodesics
have distinct lengths \cite{Abr68} : as a result, the length spectrum
$\scL$ contains exponentially many points in an interval $I(T)$
of fixed size centered at $T\to+\infty$. This suggests naturally
to study the distribution of nearby lengths in $\scL$ in this asymptotic
limit : see \cite{Dol98,Ana00,PolSha01,PetSto12_closed} and references
given there for refinements of the counting estimate \eqref{eq: Margulis}. 

The length spectrum is also of particular interest when studying the
connexions between the geometry of $M$ and the spectrum of its Laplace-Beltrami
operator $\Delta_{g}$ for the metric $g$ : if closed geodesics are
isolated and non-degenerate in the sense of Morse, it is well known
since \cite{CdV73,Chaz74,DuGu75} that the singular support of the
tempered distribution $\mbox{Trace}(\e^{\i t\sqrt{\Delta_{g}}})$
is given by the period spectrum $\scP=\{k\ell,\ \ell\in\scL,k\in\N\}$.
Such distributional traces have been used widely since then in various
direct and inverse spectral problems. 

\medskip{}

A box principle used with equation \eqref{eq: Margulis} shows that
in $\scL\cap I(T)$, there are many gaps of size at least $\e^{-h_{\top}T}$
when $T$ is large. However, nothing excludes a priori some over-exponential
clustering in the length spectrum : in this case it would not be possible
to control the gaps uniformly from below. Some interesting results
concerning this question have been obtained by Dolgopyat and Jakobson
in \cite{DolJak16_gaps}. To precise the problem of estimating the
gaps in $\scL$ from below, we will say that the length spectrum is
\emph{exponentially separated }if there exist $C,\nu>0$ with the
property that 
\begin{equation}
\forall\ell,\ell'\in\scL\ \mbox{with}\ \ell\neq\ell',\qquad|\ell-\ell'|\geq C\e^{-\nu\max(\ell,\ell')}.\label{eq: expo separation}
\end{equation}
In \cite{DolJak16_gaps} it is established that the length spectrum
of hyperbolic manifolds is exponentially separated as soon as the
fundamental group have algebraic generators. The authors note that
it is always the case for finite volume hyperbolic manifolds of dimension
$n\geq3$, see \cite{GarRag70}. They also show that for compact hyperbolic
surfaces, \eqref{eq: expo separation} is true for a dense set in
the corresponding Teichmüller space.

In variable negative curvature, Theorem 4.1 in \cite{DolJak16_gaps}
gives a rather different picture : \eqref{eq: expo separation} is
false for a $G_{\delta}$-dense set of metrics for the $C^{k}$-topology,
$k>3$. Hence metrics with a length spectrum which is not exponentially
separated are topologically generic, and then dense.

In this note, we provide a complementary result : in variable negative
curvature, the set of metrics satisfying \eqref{eq: expo separation}
is \emph{also }dense for the $C^{k}$-topology, with $k\geq2$. To
state our result more precisely, we assume that $(M,g_{0})$ is a
closed manifold of dimension $\geq2$ with negative curvature, and
that $g_{0}$ is of class $C^{r}$ with $r\geq2$. Write the bounds
of the sectional curvature $K(g_{0})$ as 
\[
-k_{0}\leq K(g_{0})\leq-k_{1},\qquad k_{0},k_{1}>0.
\]
Fix some integer $k\in[2,r]$, $ $ and for $\eps_{0}>0$ define the
set of Riemannian metrics 
\[
\cM_{k}(\eps_{0})\defeq\{g\,:\ \|g-g_{0}\|_{C^{k}}<\eps_{0}\}.
\]
Let us define also 
\begin{equation}
-K_{0}=\inf_{g\in\cM_{k}(\eps_{0})}K(g),\quad-K_{1}=\sup_{g\in\cM_{k}(\eps_{0})}K(g),\quad\kappa=\sqrt{|K_{0}|},\quad h=h_{\top}\sqrt{1+\eps_{0}}.\label{eq: def kappa}
\end{equation}
The constant $\eps_{0}$ will always be fixed and chosen small enough
so that $K_{1}>0$, which is possible as the sectional curvature is
a continuous function of the metric in the $C^{2}$-topology. Our
main result can be stated as follows: 
\begin{thm}
\label{thm: split} Let $M$ be a closed Riemannian manifold equipped
with a $C^{r}$ metric $g_{0}$ with negative curvature, $r\geq2$.
Take $k\in[2,r]$ and $\eps_{0}>0$ as above. For some arbitrary $\eps>0$,
define 
\[
\nu_{k}=(k+2)h+(k+1)\kappa+\eps.
\]
There is a metric $g\in\cM_{k}(\eps_{0})$ and $C=C(g_{0},\eps_{0},\eps,k)>0$
such that for any distinct $\ell,\ell'\in\scL_{g}$, 
\[
|\ell-\ell'|\geq C\e^{-\nu_{k}\max(\ell,\ell')}.
\]

\end{thm}
 Let us point out that if $k$ is fixed with $k<r$, the metric $g$
given by our proof of Theorem \ref{thm: split} is not $C^{k'}$ for
$r\geq k'>k$. This is due to the fact that $g$ is obtained as a
limit of $C^{r}$ metrics $(g_{n})_{n\in\N}$ with diverging $C^{k'}$
norm if $k'>k$. Improving the regularity of $g$ without increasing
$\nu_{k}$ , if possible, would probably require a new approach than
the one presented here.

\medskip{}

Some fine control over the gaps in the length spectrum is in particular
important for two well-known spectral problems : precise error terms
for Weyl's laws on a negatively curved manifold \cite{JakPol07,JaPoTo07_low},
and lower bounds for the asymptotic distribution of resonances of
the Laplacian on non-compact manifolds with hyperbolic trapped sets
\cite{GuiZwo99_Tr,Pet02_low}. In both cases, one has to study accurately
the contributions of potentially many periodic orbits of the geodesic
flow to a semiclassical trace formula, which involves the distributional
trace we alluded to above. Periodic orbits are subsets of the unit
tangent bundle $T^{1}M$, and it turns out that they can be exponentially
isolated there : this is a consequence of the expansivity of the geodesic
flow, see also Section \ref{sec:Separation} below. This allows to
overcome a lack of control of the gaps in the length spectrum to establish
a trace formula by microlocalizing near each closed geodesic in the
unit cotangent bundle -- see \cite{JaPoTo07_low}. Such formulæ involve
sums of the form 
\[
\sum_{\gamma}\sum_{k\in\N}A_{\gamma,k}\cos(\lambda k\ell(\gamma))\one_{J(T)}(k\ell(\gamma)),\qquad T\leq\cO(\log\lambda),\ \lambda\to+\infty,
\]
where $A_{\gamma,k}>0$ is related to the Poincaré map of $\gamma$,
and here $J(T)\subset[0,T]$ is some interval. Obtaining lower bounds
for the above sum is one of the main difficulties that stand in the
way of improved Weyl laws for eigenvalues (or resonances). Such estimates
requires in particular the control of the oscillating terms which
are directly connected with the distribution of the gaps in the length
spectrum.

\section{Separation of orbits in phase space\label{sec:Separation}}

Let us start this section by gathering some standard facts and notations
about the geodesic flow $\Phi^{t}:T^{1}M\to T^{1}M$. In the following,
closed geodesics will always be \emph{unoriented}. The set of closed
geodesics for $(M,g)$ with $g\in\cM(\eps_{0})$ will be denoted by
$\scC_{g}$. If $I\subset\R_{+}$ is an interval, we will also make
use of the sets 
\[
\scC_{g}(I)=\{\gamma\in\scC_{g}:\ \ell(\gamma)\in I\},\qquad\scC_{g}(T)=\scC_{g}([0,T]),\ T>0.
\]
If $\rho\in T^{1}M$ and $t\in\R$, we write $\rho(t)\defeq\Phi^{t}(\rho)$.
By $\bga$, we will denote one of the two possible lifts of $\gamma$
in $T^{1}M$, namely $(\gamma(t),\pm\dot{\gamma}(t))_{t\in\R}$ where
$t\mapsto\gamma(t)$ is any parameterization of $\gamma$ by arc-length. 

We equip $T^{1}M$ with the Sasaki metric and write $d_{S}$ for the
induced distance. A fundamental estimate coming from the analysis
of the Green subbundles \cite{Ebe73_Anosov} together with Rauch's
comparison theorem and \eqref{eq: def kappa} yields to 
\begin{equation}
d_{S}(\Phi^{t}(\rho),\Phi^{t}(\rho'))\leq\kappa_{0}\e^{\kappa|t|}d_{S}(\rho,\rho'),\qquad\kappa_{0}\defeq\sqrt{1+\kappa},\label{eq: Bowen growth}
\end{equation}
for all $\rho,\rho'\in T^{1}M$ and $t\in\R$. 

For any $\gamma_{0}\in\scC_{g_{0}}$, the curve $\gamma_{0}\subset M$
is still a closed path for $g$ : since $g\in\cM_{k}(\eps_{0})$ is
negatively curved, there is a unique $g$-geodesic $\gamma$ in the
same free homotopy class. Hence, there is a well defined, bijective
map 
\begin{equation}
f_{g_{0}\to g}:\begin{cases}
\scC_{g_{0}}\to\scC_{g}\\
\gamma_{0}\to\gamma.
\end{cases}\label{eq: bijection of geodesics}
\end{equation}
Furthermore, the $g$-length of $\gamma=f_{g_{0}\to g}(\gamma_{0})$
satisfies 
\begin{equation}
\frac{\ell_{g_{0}}(\gamma_{0})}{\sqrt{1+\eps_{0}}}\leq\ell_{g}(\gamma)\leq\ell_{g_{0}}(\gamma_{0})\sqrt{1+\eps_{0}}.\label{eq: length perturb}
\end{equation}
Also, by a theorem of Klingenberg, $K(g)<0$ implies that the injectivity
radius $r_{\inj}(g)$ is determined by the shortest closed geodesic,
\begin{equation}
r_{\inj}(g)=\frac{1}{2}\min\{\ell(\gamma),\ \gamma\in\scC_{g}\}.\label{eq: mini length}
\end{equation}
Hence \eqref{eq: length perturb} implies 
\[
\forall g\in\cM_{k}(\eps_{0}),\quad r_{\inj}(g)\geq\frac{1}{\sqrt{1+\eps_{0}}}r_{\inj}(g_{0}).
\]
In the following, we fix a number $r_{m}>0$ such that $r_{m}<(1+\eps_{0})^{-1/2}r_{\inj}(g_{0})$.
Finally, if $\gamma\in\scC_{g}$, we define the tubular open neighborhood
of a lift $\bga\subset T^{1}M$ by 
\[
\Theta_{\bga}^{\epsilon}\defeq\{\rho\in T^{1}M,\ d_{S}(\rho,\bga)<\epsilon\}.
\]
The purpose of this section is to prove the following exponential
separation result in $T^{1}M$ : 
\begin{prop}
\label{prop: Phase-space-separation} Let $g\in\cM_{k}(\eps_{0})$.
There is $\epsilon_{0}>0$ depending only on $g_{0},\eps_{0}$ such
that any distinct $\beta,\gamma\in\scC_{g}(T)$ satisfy 
\[
\Theta_{\bbe}^{\epsilon_{0}\e^{-2\kappa T}}\cap\Theta_{\bga}^{\epsilon_{0}\e^{-2\kappa T}}=\emptyset\,,
\]
where $\bbe,\bga$ denote any lifts of $\beta,\gamma$ in $T^{1}M$. 
\end{prop}
We begin with a lemma where the lengths of the geodesics are restricted
to a fixed interval. 
\begin{lem}
\label{lem: Separation Orbits} Let $g\in\cM_{k}(\eps_{0})$. For
any distinct $\gamma,\gamma'\in\scC_{g}(T)$ with $|\ell(\gamma)-\ell(\gamma')|<r_{m}/2$,
we have 
\[
\Theta_{\bga}^{\delta\e^{-\kappa T}}\cap\Theta_{\bga'}^{\delta\e^{-\kappa T}}=\emptyset\,,
\]
where $\bga,\bga'$ denote any lifts of $\gamma,\gamma'$ in $T^{1}M$
and $\delta=\frac{r_{m}}{2\kappa_{0}}$. \end{lem}
\begin{proof}
We borrow the method from \cite{JaPoTo07_low} and argue by contradiction.
Let $\gamma,\gamma'\in\scC_{g}(T)$ be two distinct geodesics with
$|\ell(\gamma)-\ell(\gamma')|<r_{m}/2$ and take two lifts 
\[
\bga(t)=\Phi^{t}(z),\qquad\bga'(t)=\Phi^{t}(z'),\qquad z,z'\in T^{1}M,\ t\in\R.
\]
Suppose that $\Theta_{\bga}^{\delta\e^{-\kappa T}}\cap\Theta_{\bga'}^{\delta\e^{-\kappa T}}\neq\emptyset$
where $\delta$ is as in the lemma. Without loss of generality, we
can assume that $d_{S}(z,z')<\delta\e^{-\kappa T}$. From \eqref{eq: Bowen growth},
we have $d_{S}(z(t),z'(t))<r_{m}/2$ for $|t|\leq T$, and this implies
$d_{g}(\gamma(t),\gamma'(t))<r_{m}/2$. If we reparameterize $\gamma$
by setting 
\[
\beta(s)\defeq\gamma(ls/l'),\quad0\leq s\leq l',
\]
we have by the triangle inequality 
\[
\forall t\in[0,l'],\quad d_{g}(\gamma'(t),\beta(t))\leq r_{m}/2+(t-t\frac{l}{l'})<r_{m}/2+r_{m}/2.
\]
Since $r_{m}<r_{\inj}(g)$, this means that there is an homotopy by
geodesic segments between the closed curves $\gamma'$ and $\beta$.
But $\beta$ being a reparameterization of $\gamma$, this implies
that we would have two distinct closed geodesics in the same free
homotopy class. 
\end{proof}
\emph{Proof of Proposition \ref{prop: Phase-space-separation}. }The
preceding lemma can be generalized for $\beta,\gamma\in\scC_{g}(T)$
as follows. Before, let us recall a classical fact for hyperbolic
flows which is at the basis of the Anosov Closing and Shadowing lemmas.
Denote the dynamical ball 
\[
B_{S}^{t}(\rho_{0},r)\defeq\{\rho\in T^{1}M:\sup_{0\leq s\leq t}d_{S}(\Phi^{s}(\rho),\Phi^{s}(\rho_{0}))<r\}.
\]
Suppose that $\rho_{0}\in T^{1}M$ is a periodic point of period $\tau$
for the geodesic flow. There exist a constant $\sigma(g)>0$, depending
continuously on the metric in the $C^{2}$-topology such that if $\Sigma_{\sigma(g)}(\rho_{0})$
is a transversal section of the flow at $\rho_{0}$ with radius $\sigma(g)$,
then for all $\rho\in\Sigma_{\sigma(g)}(\rho_{0})\cap B_{S}^{\tau}(\rho_{0},\sigma(g))$
we have $\Phi^{\tau}(\rho)\neq\rho$. Namely, $\rho_{0}$ is the unique
periodic point of period $\tau$ in $\Sigma_{\sigma(g)}(\rho_{0})$
sufficiently close to $\rho_{0}$ for the dynamical distance. This
follows from a standard fixed point argument using hyperbolicity and
a sequence of quasi-linear $C^{1}$ approximations of the flow near
the orbit $\rho_{0}(t)$ for $0\leq t\leq\tau$, see \cite[Chapters 6 and 18]{KaHa95}.

Since we consider $g\in\cM_{k}(\eps_{0})$ with $k\geq2$, we can
define 
\[
\sigma\defeq\inf_{g\in\cM_{k}(\eps_{0})}\sigma(g)>0
\]
to get a constant uniform with respect to $g\in\cM_{k}(\eps_{0})$.

Consider now distinct $\beta,\gamma\in\scC_{g}(T)$ and assume for
instance that 
\[
T\geq l\defeq\ell(\gamma)\geq\ell(\beta)\defeq l'.
\]
Fix some $\epsilon_{0}>0$ to be chosen sufficiently small later (in
function of $g_{0},\eps_{0}$), and take two lifts $\bbe,\bga\subset T^{1}M$
of these closed geodesics. Suppose that there are $z\in\bga$, $z'\in\bbe$
such that $d_{S}(z,z')<\epsilon_{0}\e^{-2\kappa T}$. From \eqref{eq: Bowen growth},
we have 
\[
d_{S}(\Phi^{l'}(z),z')=d_{S}(\Phi^{l'}(z),\Phi^{l'}(z'))\leq\epsilon_{0}\kappa_{0}\e^{-\kappa T}.
\]
Hence the point $\Phi^{t}(z)$ stays exponentially close to $\Phi^{t}(z')$
for $t\in[0,l']$. Moreover, 
\begin{equation}
d_{S}(\Phi^{l'}(z),z)\leq d_{S}(\Phi^{l'}(z),z')+d_{S}(z',z)\leq2\epsilon_{0}\kappa_{0}\e^{-\kappa T}.\label{eq: return map}
\end{equation}
Let $\Sigma_{\sigma}(z)$ be a Poincaré section of the flow of size
$\sigma$ centered at $z$. Without loss of generality, we can assume
that $z'\in\Sigma_{\sigma}(z)$. Equation \eqref{eq: return map}
implies that for $\epsilon_{0}$ small enough depending on $g_{0},\eps_{0}$
via $\kappa,\sigma$, the point $\Phi^{l'}(z)$ belongs to a flow
box with section $\Sigma_{\sigma}(z)$ : there is a unique $s\in\R$
such that 
\[
\Phi^{l'+s}(z)=\zeta\in\Sigma_{\sigma}(z)
\]
 where $d_{S}(\zeta,z)\leq C\epsilon_{0}\e^{-\kappa T}$, $|s|\leq C\epsilon_{0}\e^{-\kappa T}$.
Here $C=C(g_{0},\eps_{0})\geq1$ can be chosen uniformly for $g\in\cM_{k}(\eps_{0})$
since the distance $d_{S}$ varies continuously with $g$.

But now, $\zeta$ and $z$ are on the same periodic orbit, so they
are periodic with period $l$ and both belong to $\Sigma_{\sigma}(z)$.
Also, using \eqref{eq: Bowen growth} one more time, we have 
\[
d_{S}(\zeta,z)\leq C\epsilon_{0}\e^{-\kappa T}\ \Rightarrow\ \sup_{[0,l]}d_{S}(\Phi^{t}(\zeta),\Phi^{t}(z))\leq C\epsilon_{0}\kappa_{0}.
\]
Therefore, if $\epsilon_{0}$ is small enough, the right hand side
of the previous equation is $<\sigma$ and we have $\zeta\in B_{S}^{T}(z,\sigma)\cap\Sigma_{\sigma}(z)$.
From the remark above, this forces $\zeta=z$. Up to shrink $\epsilon_{0}$
further, this finally implies 
\[
s=l-l'\leq C\epsilon_{0}\e^{-\kappa T}<r_{m}/2.
\]
Hence we can apply Lemma \ref{lem: Separation Orbits} , in contradiction
with $d_{S}(z,z')<\epsilon_{0}\e^{-2\kappa T}<\delta\e^{-\kappa T}$.

\section{Almost intersections of closed geodesics\label{sec:Almost-intersections}}

In this section we prove the key tool needed in the proof of Theorem
\ref{thm: split}. If $x\in M$ and $r>0$, we denote by $B_{g}(x,r)\subset M$
an open ball of radius $r$ centered at $x$ for the metric $g$. 
\begin{prop}
\label{thm: conformal place} Let $g\in\cM_{k}(\eps_{0})$. Fix $\eps>0$
and take $\alpha=2\kappa+h+\eps$. There is $T_{0}>0$ depending only
on $g_{0},\eps_{0},\eps$ such that for $T\geq T_{0}$ and any $\beta\in\scC_{g}(T)$,
there exists $z\in\beta$ with 
\begin{equation}
\forall\gamma\in\scC_{g}(T)\setminus\beta,\qquad B_{g}(z,\e^{-\alpha T})\cap\gamma=\emptyset,\label{eq: conformal place others}
\end{equation}
 and moreover, $B_{g}(z,\e^{-\alpha T})\cap\beta$ consists in a unique
geodesic segment of $\beta$ centered at $z$. 
\end{prop}
The main idea is as follow : since two closed geodesics $\beta,\gamma\in\scC_{g}(T)$
are separated by $\epsilon_{0}\e^{-2\kappa T}$ when lifted in $T^{1}M$,
it means that if $\beta$ and $\gamma$ are very close somewhere in
$M$, the angle ``between'' their tangent vectors must be bounded
below. It follows that the two geodesics cannot stay close to each
other in $M$ for a long time. Proposition \ref{prop: Covering segments}
below is a quantitative version of this observation.

\subsection{Local divergence of orbits separated in phase space}

The next proposition is a simple application of the Toponogov comparison
theorem \cite{Kar89} to study the local divergence of two geodesics
close from each other in $M$, but separated when lifted in $T^{1}M$. 
\begin{lem}
\label{lem : triangle} Assume that $g\in\cM_{k}(\eps_{0})$. Let
$\alpha\geq2\kappa+h$ and $(x,\xi),\,(x',\xi')\in T^{1}M$. Suppose
that 
\begin{equation}
d_{g}(x,x')<\e^{-\alpha T},\qquad d_{S}((x,\pm\xi),(x',\xi'))\geq\epsilon_{0}\e^{-2\kappa T}\,,\label{eq: hyp separation}
\end{equation}
and denote by $I_{x}$ the following geodesic segment : 
\[
I_{x}=\{\pi\circ\Phi^{t}(x,\xi),\,|t|\leq r_{m}/2\}\subset M,\qquad\pi:T^{1}M\to M.
\]
There are $T_{0},C_{1},C_{2}>0$ depending only on $g_{0},\eps_{0}$
such that if $T\geq T_{0}$ and $|t'|\leq r_{m}/2$, 
\begin{equation}
d_{g}(\pi\circ\Phi^{t'}(x',\xi'),I_{x})\geq\max\{C_{1}|t'|\e^{-2\kappa T}-C_{2}\e^{-\alpha T},0\}.\label{eq: distance segment perturb}
\end{equation}
\end{lem}
\begin{proof}
Write $\rho=(x,\xi),\rho'=(x',\xi')$. We consider first the case
where $x=x'$. Since $g$ has negative curvature, the ball $B_{g}(x,r_{m}/2)$
is convex and the geodesic triangle defined by the points 
\[
\{x,\pi(\rho(t)),\pi(\rho'(t'))\}
\]
 is entirely contained in this ball for $|t|,|t'|\leq r_{m}/2$. Write
$\measuredangle(\xi,\xi')\in[0,\pi]$ for the angle between $\xi$
and $\xi'$ measured with the metric $g$. From \eqref{eq: hyp separation},
we can assume without loss of generality (up to change $\xi\to-\xi$)
that $\measuredangle(\xi,\xi')\in[0,\pi/2]$. From the property of
the Sasaki metric, we have precisely $d_{S}(\rho,\rho')=|\measuredangle(\xi,\xi')|$
since $x=x'$. Hence \eqref{eq: hyp separation} yields to 
\begin{equation}
\epsilon_{0}\e^{-2\kappa T}\leq d_{S}(\rho,\rho')=|\measuredangle(\xi,\xi')|.\label{eq: lower bound angle}
\end{equation}
Let now $\{A,B,C\}$ be an Euclidian triangle with $AB=|t|$, $AC=|t'|$
and 
\[
\measuredangle(AB,AC)=\measuredangle(\xi,\xi').
\]
Since $|t|,|t'|\leq r_{m}/2$, the Toponogov comparison theorem for
negative curvature together with \eqref{eq: lower bound angle} implies
that

\begin{equation}
\frac{1}{2}|t'|\epsilon_{0}\e^{-2\kappa T}\leq|t'\sin\measuredangle(\xi,\xi')|\leq BC\leq d_{g}(\pi(\rho(t)),\pi(\rho'(t'))).\label{eq: Toponogov}
\end{equation}
Note that minimizing on $t\in[-r_{m}/2,r_{m}/2]$ gives \eqref{eq: distance segment perturb}
with $C_{2}=0$. 

We move on now to the case where $x\neq x'$ with $d_{g}(x,x')<\e^{-\alpha T}$.
Consider in $T^{1}M$ the curve $(c(t),v(t))_{0\leq t\leq d_{g}(x,x')}$
where $c(t)$ is a minimizing geodesic connecting $x'$ to $x$ and
$v(t)$ is the parallel transport of $\xi'$ along $c(t)$ for $0\leq t\leq d_{g}(x,x')$.
Define $\zeta=v(d_{g}(x,x'))\in T_{x}^{1}M$ and set 
\[
\ti\rho\defeq(x,\zeta).
\]
The vector $v(t)$ belongs to the horizontal distribution along $c(t)$,
and $\pi(\ti\rho)=\pi(\rho)=x$ : by the definition of the Sasaki
metric, this gives 
\[
d_{S}(\rho',\ti\rho)=d_{g}(x,x'),\qquad d_{S}(\rho,\ti\rho)=|\measuredangle(\xi,\zeta)|.
\]
From \eqref{eq: hyp separation}, the triangle inequality and the
fact that $\alpha\geq2\kappa+h$, we get 
\[
\e^{-2\kappa T}(\epsilon_{0}-\e^{-hT})\leq|\measuredangle(\xi,\zeta)|.
\]
Hence for $T$ large enough depending only on $g_{0},\eps_{0}$ via
$\epsilon_{0},h$, we have $|\measuredangle(\xi,\zeta)|\geq\epsilon_{0}\e^{-2\kappa T}/2$.
From this observation, we proceed as for \eqref{eq: Toponogov} :
assuming without loss of generality that $\measuredangle(\xi,\zeta)\in[0,\pi/2]$
we have 
\begin{eqnarray*}
\frac{1}{4}|t'|\epsilon_{0}\e^{-2\kappa T} & \leq & d_{g}(\pi(\rho(t)),\pi(\ti\rho(t')))\\
 & \leq & d_{g}(\pi(\rho(t)),\pi(\rho'(t')))+d_{g}(\pi(\rho'(t')),\pi(\ti\rho(t')))\\
 & \leq & d_{g}(\pi(\rho(t)),\pi(\rho'(t')))+\kappa_{0}\e^{\kappa|t'|}\e^{-\alpha T}.
\end{eqnarray*}
But $|t'|\leq r_{m}/2$, so 
\[
\frac{\epsilon_{0}}{4}|t'|\e^{-2\kappa T}-\kappa_{0}\e^{\kappa r_{m}/2}\e^{-\alpha T}\leq d(\pi(\rho(t)),\pi(\rho'(t'))),
\]
and we conclude the proof by minimizing on $t\in[-r_{m}/2,r_{m}/2]$
and setting $C_{1}=\epsilon_{0}/4$, $C_{2}=\kappa_{0}\e^{\kappa r_{m}/2}$. 
\end{proof}
Given two geodesics $\beta,\gamma\in\scC_{g}(T)$, the preceding proposition
will enable us to estimate the total length of pieces of these geodesics
which are close to each other.

\subsection{Intersections and almost-intersections }

Roughly speaking, finding $z$ in Proposition \ref{thm: conformal place}
can be done if we are able to find a sufficiently large piece of $\beta\in\scC_{g}(T)$
that ``avoids'' both all other geodesics in $\scC_{g}(T)$, and
all the rest of $\beta$ . This suggests to study the situations where
two given closed geodesics intersect, or more generally come close
to each other.

Let $\beta,\gamma\in\scC_{g}(T)$ be two closed geodesics. It is essentially
well known that their intersection number $i(\beta,\gamma)\in\N$
grows quadratically with $T$ : 
\begin{equation}
i(\beta,\gamma)=\cO_{r_{\inj(g)}}\left(T\right)^{2}.\label{eq: intersection numbre}
\end{equation}
This is a consequence of the fact that if $x$ and $x'$ are two intersection
points such that $x\xrightarrow{\gamma}x'$ is a segment, a loop 
\[
x\xrightarrow{\gamma}x'\xrightarrow{\beta}x
\]
 satisfies $\ell(x'\xrightarrow{\beta}x)\geq r_{\inj}(g)$. The estimate
\eqref{eq: intersection numbre} follows by cutting $\gamma$ into
pieces of length $<r_{\inj}(g)$ and applying the above property repeatedly.
Intersections of two closed geodesics are special situations where
small tubular neighborhoods of the geodesics intersect, and the above
remark shows that the topology of the manifold excludes that too many
intersections points can be themselves close to each other. 

In dimension $3$ or more, intersections and self-intersections are
marginal, so we would like to generalize \eqref{eq: intersection numbre}
to ``almost-intersections'', namely regions in $M$ where two geodesics
are close to each other but without necessarily intersecting. Of course
points of almost-intersection will not be countable, so we will have
to consider small segments. 

Before continuing further, for $\epsilon>0$ let us denote by 
\[
\theta_{\gamma}^{\epsilon}\defeq\{x\in M\ |\ d_{g}(x,\gamma)<\epsilon\}
\]
an open tubular neighborhood of $\gamma\in\scC_{g}$ of size $\epsilon$.
 The next proposition allows to control the size of $\beta\cap\theta_{\gamma}^{\epsilon}$
when $\beta,\gamma\in\scC_{g}(T)$ : 
\begin{prop}
\label{prop: Covering segments}Let $\beta,\gamma\in\scC_{g}(T)$.
Fix $\alpha\geq2\kappa+h$ and set 
\[
\epsilon\defeq\e^{-\alpha T}.
\]
There are $T_{0},C_{3}>0$ depending only on $g_{0},\eps_{0}$ such
that for $T>T_{0}$, if the set $\beta\cap\theta_{\gamma}^{\epsilon}$
is not empty, it can be covered by a finite number of geodesic segments
$\{J_{1},\dots,J_{n(\beta,\gamma)}\}\subset\beta$ with the following
properties :

(i) $0\leq n(\beta,\gamma)\leq4\left(\frac{T}{r_{m}}\right)^{2},$

(ii) $|J_{i}|\leq C_{3}\e^{-(\alpha-2\kappa)T}$ for all $1\leq i\leq n(\beta,\gamma)$. 
\end{prop}
Before giving the proof of this proposition, we need some preliminary
results. We fix two parameterizations of $\beta,\gamma$ by arc-length
and define the continuous maps 
\[
D_{\beta,\gamma}:\begin{cases}
[0,\ell(\beta)]\times[0,\ell(\gamma)]\to\R_{+}\\
(s,t)\mapsto d_{g}(\beta(s),\gamma(t))\,.
\end{cases}G_{\beta,\gamma}:\begin{cases}
[0,\ell(\beta)]\times[0,\ell(\gamma)]\to M\times M\\
(s,t)\mapsto(\beta(s),\gamma(t))\,.
\end{cases}
\]
Consider an open connected component $U\subset\R^{2}$ of $D_{\beta,\gamma}^{-1}(]0,\epsilon[)$.
Such a set $U$ exists since we assumed $\beta\cap\theta_{\gamma}^{\epsilon}\neq\emptyset$,
furthermore, there is $(s_{u},t_{u})\in U$ which is a local minimum
of $D_{\beta,\gamma}|_{U}$ . Assume that $T$ is large enough so
that $\e^{-\alpha T}\leq r_{m}/2$. By the convexity of the distance
function in negative curvature, for $s,t$ such that $0<|t|\leq r_{m}/2$
and $0<|s|\leq r_{m}/2$, we have 
\[
d_{g}(\beta(s_{u}+s),\gamma(t_{u}+t))>d_{g}(\beta(s_{u}),\gamma(t_{u})).
\]
It follows that the local minima of $D_{\beta,\gamma}$ are isolated,
and the total number of local minima of this map is at most
\[
\frac{\ell(\beta)\times\ell(\gamma)}{(r_{m}/2)^{2}}\leq4\left(\frac{T}{r_{m}}\right)^{2}.
\]
If $(s,t)\in[0,\ell(\beta)]\times[0,\ell(\gamma)]$ is a local minimum
of $D_{\beta,\gamma}$ and $(x,y)=G_{\beta,\gamma}(s,t)$, we say
that $(x,y)$ is an \emph{almost-intersection} of $\beta$ and $\gamma$. 

Consider now an almost intersection $(x,y)\in\beta\times\gamma$,
and let us shift the origin of the parameterizations of $\beta$ and
$\gamma$ so that we can write 
\[
(x(0),\dot{x}(0))=(x,\xi),\qquad(y(0),\dot{y}(0))=(y,\eta).
\]
Define the open segment $I_{x}=\{\pi\circ\Phi^{t}(x,\xi)\::|t|\leq r_{m}/2\}\subset\beta$
and its open $\epsilon$-neighborhood in $M$ by 
\[
\theta_{\beta}^{\epsilon}(I_{x})=\{z\in M:\ d_{g}(z,I_{x})<\epsilon\}.
\]
Define similarly $I_{y}$ and $\theta_{\gamma}^{\epsilon}(I_{y})$. 
\begin{lem}
\label{lem: almost intersection} Let $(x,y)$ be an almost intersection
as above, and fix $\alpha\geq2\kappa+h$. For $T>0$ sufficiently
large depending only on $g_{0},\eps_{0}$, there are open segments
$J(x,y)\subset\beta$ containing $x$ and $J(y,x)\subset\gamma$ containing
$y$ such that : 

(i) $I_{x}\cap\theta_{\gamma}^{\epsilon}(I_{y})\subset J(x,y)$ and
$I_{y}\cap\theta_{\beta}^{\epsilon}(I_{x})\subset J(y,x)$, 

(ii) $\max\{|J(x,y)|,|J(y,x)|\}\leq C_{3}\e^{-(\alpha-2\kappa)T}$,
where $C_{3}$ depends only on $g_{0},\eps_{0}$. \\
In other words, 
\[
\forall z\in I_{x}\setminus J(x,y),\quad d_{g}(z,I_{y})\geq\epsilon,
\]
 and the symmetric property is true by exchanging the roles of $x$
and $y$.\end{lem}
\begin{proof}
Since Proposition \ref{prop: Phase-space-separation} ensures that
$\Theta_{\bga}^{\epsilon_{0}\e^{-2\kappa T}}\cap\Theta_{\bbe}^{\epsilon_{0}\e^{-2\kappa T}}=\emptyset$
for any lifts $\bga,\bbe$ in $T^{1}M$, this implies 
\[
d_{S}((x,\pm\xi),(y,\eta))\geq\epsilon_{0}\e^{-2\kappa T}.
\]
On the other hand, $d_{g}(x,y)<\epsilon$, so we are in position to
apply Proposition \ref{lem : triangle}, which gives readily 
\begin{equation}
d_{g}(x(t),I_{y})\geq C_{1}|t|\e^{-2\kappa T}-C_{2}\e^{-\alpha T},\quad t\in[-r_{m}/2,r_{m}/2],\label{eq: distance J}
\end{equation}
and the same equation holds true by exchanging the roles of $x$ and
$y$. Define $C_{3}=2C_{1}^{-1}(1+C_{2})$ and take $T\geq T_{0}$
where $C_{3}\e^{-hT_{0}}<r_{m}$. In this case, we have 
\[
r_{m}/2\geq|t|\geq C_{1}^{-1}(1+C_{2})\e^{(2\kappa-\alpha)T}=\frac{C_{3}}{2}\e^{-(\alpha-2\kappa)T}\Rightarrow d_{g}(x(t),I_{y})\geq\epsilon=\e^{-\alpha T},
\]
Therefore, there is a (maximal) non-empty open interval $]t_{x}^{-},t_{x}^{+}[\subset[-r_{m}/2,r_{m}/2]$
with $|t_{x}^{+}-t_{x}^{-}|\leq C_{3}\e^{-(\alpha-2\kappa)T}$ such
that 
\[
x(t)\in\theta_{\gamma}^{\epsilon}(I_{y})\ \Rightarrow\ t\in]t_{x}^{-},t_{x}^{+}[.
\]

We can define an interval $]t_{y}^{-},t_{y}^{+}[$ in the same way
by permuting the roles of $x$ and $y$. The open segments with the
desired properties are precisely 
\[
J(x,y)\defeq\{x(t),\ t_{x}^{-}<t<t_{x}^{+}\},\qquad J(y,x)\defeq\{y(t),\ t_{y}^{-}<t<t_{y}^{+}\}.
\]

\end{proof}
\emph{Proof of Proposition }\ref{prop: Covering segments}. Call 
\[
(s_{1},t_{1}),\dots,(s_{n(\beta,\gamma)},t_{n(\beta,\gamma)}),\quad n(\beta,\gamma)\leq4\left(\frac{T}{r_{m}}\right)^{2}
\]
the local minima of the function $D_{\beta,\gamma}$, and write $(x_{i},y_{i})=G_{\beta,\gamma}(s_{i},t_{i})$
the almost-intersections identified with these local minima via the
parameterizations of the closed geodesics. We just need to check that
if $(x,y)\in\beta\times\gamma$ is such that 
\[
d_{g}(x,y)<\epsilon,
\]
then there is $i\in[1,n(\beta,\gamma)]$ such that $x\in J(x_{i},y_{i})$
and $y\in J(y_{i},x_{i})$ where the intervals are given by the preceding
lemma. This is clear if $(x,y)$ is an almost-intersection. Otherwise
let $(s_{x},t_{y})$ be such that $G_{\beta,\gamma}(s_{x},t_{x})=(x,y)$,
and $U\subset[0,\ell(\beta)]\times[0,\ell(\gamma)]$ be the connected
component of $D_{\beta,\gamma}^{-1}(]0,\epsilon[)$ containing $(s_{x},t_{y})$.
$U$ contains a local minimum $(\ti s,\ti t)$ of $D_{\beta,\gamma}$,
and since it is arc-connected (as it is locally), there is a continuous
path 
\[
f:[0,1]\to[0,\ell(\beta)]\times[0,\ell(\gamma)]
\]
 joining $(\ti s,\ti t)$ to $(s_{x},t_{y})$ which is fully contained
in $U$, namely : 
\[
G_{\beta,\gamma}\circ f(0)=(\ti x,\ti y),\qquad G_{\beta,\gamma}\circ f(1)=(x,y),\qquad(\ti x,\ti y)=G_{\beta,\gamma}(\ti s,\ti t),
\]
and 
\begin{equation}
\forall t\in[0,1],\quad d_{g}(G_{\beta,\gamma}\circ f(t))<\epsilon.\label{eq: epsilon path}
\end{equation}
Let us define 
\[
\ti T\defeq[\ti s-\frac{r_{m}}{2},\ti s+\frac{r_{m}}{2}]\times[\ti t-\frac{r_{m}}{2},\ti t+\frac{r_{m}}{2}]\subset[0,\ell(\beta)]\times[0,\ell(\gamma)].
\]
Lemma \ref{lem: almost intersection} shows precisely that 
\[
(s,t)\in\ti T\setminus G_{\beta,\gamma}^{-1}\left(J(\ti x,\ti y)\times J(\ti y,\ti x)\right)\ \Rightarrow\ D_{\beta,\gamma}(s,t)\geq\epsilon.
\]
In view of \eqref{eq: epsilon path}, the continuity of $f$ and $G_{\beta,\gamma}\circ f(0)=(\ti x,\ti y)$,
this implies that 
\[
G_{\beta,\gamma}\circ f([0,1])\subset J(\ti x,\ti y)\times J(\ti y,\ti x),
\]
and this shows that $x\in J(\ti x,\ti y)$ and $y\in J(\ti y,\ti x)$.
In particular, in each connected component of $D_{\beta,\gamma}^{-1}(]0,\epsilon[)$
there is a unique local minimum of $D_{\beta,\gamma}$. 

We have just shown that if $x\in\beta$ is such that $d_{g}(x,\gamma)<\epsilon$,
there is $i\in[1,n(\beta,\gamma)]$ such that $x\in J(x_{i},y_{i})$.
Therefore, $\beta\cap\theta_{\gamma}^{\epsilon}$ is covered by
\[
U_{\beta}(\gamma)\defeq\bigcup_{i=1}^{n(\beta,\gamma)}J(x_{i},y_{i})\subset\beta.
\]
The proof of Proposition \ref{prop: Covering segments} is completed
since there are at most $4\left(T/r_{m}\right)^{2}$ terms in the
above equation, and for all $i$, $|J(x_{i},y_{i})|\leq C_{3}\e^{-(\alpha-2\kappa)T}$
from Lemma \ref{lem: almost intersection}.

\subsection{Proof of Proposition \ref{thm: conformal place} : case of distinct
geodesics. }

In this section, we establish Equation \eqref{eq: conformal place others}.
Let $\beta\in\scC_{g}(T)$ where $g\in\cM_{k}(\eps_{0})$. For $\eps>0$,
let us choose $\alpha>0$ such that 
\[
\alpha\geq2\kappa+h+\eps.
\]
We then take $T_{0}>0$ large enough so that Propositions \ref{prop: Phase-space-separation}
and \ref{prop: Covering segments} hold true for $T\geq T_{0}$. From
Proposition \ref{prop: Covering segments}, we have 
\begin{equation}
\ell(U_{\beta}(\gamma))\leq4\left(\frac{T}{r_{m}}\right)^{2}\times C_{3}\e^{-(\alpha-2\kappa)T}.\label{eq: size U b g}
\end{equation}
We now consider all closed geodesics $\gamma\neq\beta$ with $\beta\in\scC_{g}(T)$
fixed and $\gamma\in\scC_{g}(T)$. To get a uniform bound on the counting
function for closed geodesics for the metric $g$, note first that
\eqref{eq: Margulis} implies that there is $C_{0}>0$ depending only
on $g_{0}$ such that $\sharp\scC_{g_{0}}(T)\leq C_{0}T^{-1}\e^{h_{\top}T}$
for, say, $T>r_{m}$. For $\gamma_{0}\in\scC_{g_{0}}$, let $\gamma=f_{g_{0}\to g}(\gamma_{0})$.
In view of \eqref{eq: length perturb}, 
\[
\ell_{g_{0}}(\gamma_{0})>T\sqrt{1+\eps_{0}}\Rightarrow\ell_{g}(\gamma)>T,
\]
so $f_{g_{0}\to g}^{-1}(\scC_{g}(T))\subset\scC_{g_{0}}(T\sqrt{1+\eps_{0}})$
and therefore, 
\begin{equation}
\sharp\scC_{g}(T)\leq\sharp\scC_{g_{0}}(T\sqrt{1+\eps_{0}})\leq C_{0}\frac{\e^{hT}}{T}.\label{eq: Margu grossier}
\end{equation}
Setting 
\[
\bU_{\beta}\defeq\bigcup_{\gamma\in\scC_{g}(T)\setminus\beta}U_{\beta}(\gamma),
\]
we obtain using \eqref{eq: size U b g} that 
\begin{equation}
\ell(\bU_{\beta})\leq\frac{4C_{3}C_{0}}{r_{m}^{2}}T\e^{-(\alpha-2\kappa-h)T}=\cO_{g_{0},\eps_{0}}(T\e^{-\eps T})\xrightarrow{T\to\infty}0.\label{eq: total length intersection}
\end{equation}
Define 
\[
\bV_{\beta}=\{x\in\beta:\ \forall\gamma\in\scC_{g}(T)\setminus\beta,\ d_{g}(x,\gamma)\geq\epsilon\}.
\]

By construction, $\beta\setminus\bU_{\beta}\subset\bV_{\beta}$ and
$\bU_{\beta}$ is a finite union of $\cO(T\e^{hT})$ open segments.
Therefore, we see by a box principle along $\beta$ using \eqref{eq: total length intersection}
that for $T$ large enough depending only on $g_{0},\eps_{0}$ and
$\eps$, the set $\bV_{\beta}$ contains at least one segment $I$
of size 
\[
|I|\geq T^{-2}\e^{-hT}\geq2\e^{-\alpha T}.
\]
We assumed $2\e^{-\alpha T}\leq r_{m}$, so if we choose $z\in\beta$
to be the middle of $I$, then 
\begin{equation}
\forall\gamma\in\scC_{g}(T)\setminus\beta,\qquad B_{g}(z,\epsilon)\cap\gamma=\emptyset,\label{eq: avoiding point}
\end{equation}
and this shows \eqref{eq: conformal place others}. It remains to
establish that $z$ can be chosen such that the ball $B_{g}(z,\epsilon)$
also avoids $\beta$ except on a single geodesic segment of $\beta$
containing $z$.

\subsection{Almost-intersections of a single closed geodesic}

To conclude the proof of Proposition \ref{thm: conformal place},
we now indicate how the results of the preceding sections allow us
to control almost-intersections of a closed geodesic with itself. 

As above, we take $\alpha\geq h+2\kappa+\eps$ for some $\eps>0$.
Fix some parameterization of $\beta\in\scC_{g}(T)$ by arc-length.
Let $(x,y)\in\beta\times\beta$, and call $t_{x},t_{y}$ times such
that $x=\beta(t_{x})$, $y=\beta(t_{y})$. We will say that the couple
$(x,y)$ is an almost-intersection of $\beta$ with itself if $d_{g}(x,y)<\epsilon$,
$t_{y}\neq t_{x}$ and $(t_{x},t_{y})$ is a local minimum of $D_{\beta,\beta}:(t,t')\mapsto d_{g}(\beta(t),\beta(t'))$.
In particular, either $(x,y)$ is a self-intersection of $\beta$,
or the segment joining $x$ to $y$ is not included in $\beta$. Using
as before convexity of the distance function in negative curvature,
we get that there are at most $\cO_{r_{m}}(T^{2})$ such couples of
almost-intersections, and arguments identical to those developed in
the proof of Proposition \ref{prop: Phase-space-separation} show
that if $(x,y)$ is an almost-intersection, then 
\[
d_{S}((\beta(t_{y}),\dot{\beta}(t_{y})),(\beta(t_{x}),\dot{\beta}(t_{x})))\geq\epsilon_{0}\e^{-2\kappa T}.
\]
For $z=\beta(t_{z})$, define as before 
\[
I_{z}=\{\Phi^{t}(\beta(t_{z}),\dot{\beta}(t_{z})),\ |t|\leq r_{m}/2\}.
\]
In particular, as in Lemma \ref{lem: almost intersection}, to $(x,y)$
we can associate segments $J(x,y)\subset I_{x}$ and $J(y,x)\subset I_{y}$
of $\beta$ centered at $x$ and $y$ respectively, of size $\cO_{g_{0},\eps_{0}}(\e^{-(\alpha-2\kappa)T})$
such that $\forall z\in I_{x}\setminus J(x,y)$, $d_{g}(z,I_{y})\geq\epsilon$
and symmetrically when interchanging the roles of $y$ and $x$. Exactly
as in Proposition \ref{prop: Covering segments}, we can then show
that the set 
\[
\boldsymbol{I}_{\beta}=\{z\in\beta:\ (B_{g}(z,\epsilon)\setminus I_{z})\cap\beta\neq\emptyset\}
\]
 can be covered by $\cO_{g_{0},\eps_{0}}(T^{2})$ segments of size
$\cO_{g_{0},\eps_{0}}(\e^{-(\alpha-2\kappa)T})$, for $T\geq T_{0}(g_{0},\eps_{0})$.
These segments can be added up to $\bU_{\beta}$ : exactly as previously,
we end up by a box principle with the fact that for $T$ large enough
depending now on $g_{0},\eps_{0}$ and $\eps$, there is a segment
$I\subset\beta$ such that $|I|\geq2\e^{-\alpha T}$ and if $z$ denotes
the middle of $I$, we have 
\[
\forall\gamma\in\scC_{g}(T)\setminus\beta,\quad B_{g}(z,\epsilon)\cap\gamma=\emptyset\quad\mbox{and}\quad B_{g}(z,\epsilon)\cap\beta=I_{z}\cap B_{g}(z,\epsilon).
\]
This concludes the proof of Proposition \ref{thm: conformal place}.
We end this section by a straightforward corollary : 
\begin{cor}
\label{cor: boules separees}Let $g\in\cM(\eps_{0})$, $\eps>0$ and
$\alpha\geq2\kappa+h+\eps$ be as above. For $T\geq T_{0}(g_{0},\eps_{0},\eps)$
and each $\gamma\in\scC_{g}(T)$, there is $z_{\gamma}\in\gamma$
such that 
\[
B_{g}(z_{\gamma},\epsilon/2)\cap\left(\bigcup_{\gamma'\in\scC_{g}(T)\setminus\gamma}\theta_{\gamma'}^{\epsilon/2}\right)=\emptyset,
\]
and $B_{g}(z_{\gamma},\epsilon/2)\cap\gamma$ consists in a unique
geodesic segment centered at $z_{\gamma}$. In particular, if $\gamma,\gamma'\in\scC_{g}(T)$
are distinct, then $B_{g}(z_{\gamma},\epsilon/2)\cap B_{g}(z_{\gamma'},\epsilon/2)=\emptyset.$
\end{cor}

\section{Proof of Theorem \ref{thm: split} \label{sub: last proof}}

For $T>0$ large enough, Corollary \ref{cor: boules separees} of
Proposition \ref{thm: conformal place} allow to perturb the metric
near a point of a closed geodesic in $\scC_{g}(T)$ without changing
the length of all the others in $\scC_{g}(T)$. Before exploiting
further this property to separate the length spectrum, we recall without
proof a standard fact for conformal perturbations of a given metric
$g$ near a closed geodesic. 
\begin{lem}
\label{lem: conformal dilation}Let $\gamma\in\scC_{g}$ and $x\in\gamma$.
Let also \textbf{$B_{g}(x,r_{0})$} be an open ball\textbf{ }of radius
$r_{0}<r_{\inj}(g)/2$ centered at $x$. Assume that $\gamma\in\scC_{g}$
is parametrized such that $\gamma(0)=x$ and 
\[
\gamma\cap B_{g}(x,r_{0})=\{\pi\circ\Phi^{t}(\gamma(0),\dot{\gamma}(0)),\,|t|<r_{0}\}.
\]
Fix some arbitrary $\delta\in\R$. We can choose a (radial) function
$\chi_{0}=\chi(\cdot/r_{0})\in C_{0}^{\infty}(M)$ with $\supp\chi_{0}\subset B_{g}(x,r_{0})$
such that for the metric $\ti g\defeq(1+\frac{\delta}{r_{0}}\chi_{0})^{2}g$,
the curve $\gamma$ is still a $\ti g$-closed geodesic and 
\begin{equation}
\ell_{\ti g}(\gamma)=\ell_{g}(\gamma)+\delta.\label{eq: perturb conforme}
\end{equation}
Moreover, for any $k\geq0$ we have 
\begin{equation}
\|g-\ti g\|_{C^{k}}\leq C_{k}\delta r_{0}^{-(k+1)}\|g\|_{C^{k}}\label{eq: perturb conforme Cr}
\end{equation}
where $C_{k}$ is independent of $r_{0}$. 
\end{lem}
We now pass to the proof of Theorem \ref{thm: split}. To begin with,
let us fix some $\eps>0$ and choose 
\[
\alpha=2\kappa+h+\frac{\eps}{2(k+1)}\,,\quad k\geq2.
\]
Let $T_{0}(g_{0},\eps_{0},\eps,k)>0$ be a positive number such that
Proposition \ref{thm: conformal place} with the above value of $\alpha$
is satisfied for $T\geq T_{0}$, and set $T_{n}=T_{0}+n$, $n\in\N$.

For $\nu>0$ to be defined soon below and $n\geq1$, we will say that
$\scC_{g_{n}}([T_{0},T_{n}])$ is $\nu$-separated if for all distinct
$\ell,\ell'\in\scL_{g_{n}}([T_{0},T_{n}])$, we have 
\[
|\ell-\ell'|\geq\e^{-\nu T_{n}}.
\]
We proceed iteratively : once a metric $g_{n-1}\in\cM_{k}(\eps_{0})$
($n>1$) is constructed such that $\scC_{g_{n-1}}([T_{0},T_{n-1}])$
is $\nu$-separated, we consider the next interval $]T_{n-1},T_{n}]$
and build a metric $g_{n}$ from $g_{n-1}$ such that $\scC_{g_{n}}([T_{0},T_{n}])$
is $\nu$-separated.

To do so, consider $]T_{n-1},T_{n}]$ and $n\geq1$. Denote the set
of closed geodesics with length in $]T_{n-1},T_{n}]$ by 
\[
\scC_{g_{n-1}}(]T_{n-1},T_{n}])=\{\gamma_{1},\dots,\gamma_{\mu_{n}}\}.
\]
By a box principle, in view of \eqref{eq: Margu grossier} there is
at least one couple of distinct $\ell^{-},\ell^{+}\in\scL_{g_{n-1}}(]T_{n-1},T_{n}])$
such that $\ell^{+}-\ell^{-}\geq C_{0}^{-1}\e^{-hT_{n}}.$ Setting
$\ell_{i}=\ell(\gamma_{i})$, we order the points of $\scL_{g_{n-1}}\cap]T_{n-1},T_{n}]$
so that 
\[
T_{n-1}<\ell_{1}\leq\dots\leq\ell_{m}=\ell^{-}<\ell^{+}=\ell_{m+1}\leq\dots\leq\ell_{\mu_{n}}\leq T_{n}.
\]
Let us define $\epsilon_{n}=\e^{-\alpha T_{n+1}}/2$ and choose $z_{i}\in\gamma_{i}$
according to Corollary \ref{cor: boules separees}. In particular,
$B_{g_{n-1}}(z_{i},\epsilon_{n})$ does not intersect the open $\epsilon_{n}-$neighborhood
of any $\gamma\in\scC_{g_{n-1}}(T_{n+1})\setminus\gamma_{i}$, and
$B_{g_{n-1}}(z_{i},\epsilon_{n})\cap\gamma_{i}$ consists in a single
geodesic segment centered at $z_{i}$. We are exactly in the settings
of Lemma \ref{lem: conformal dilation} : in the ball $B_{g_{n-1}}(z_{i},\epsilon_{n})$,
we dilate the metric by a factor $(1+\epsilon_{n}^{-1}\delta_{n,i}\chi_{n,i})^{2}$,
where $\chi_{n,i}$ plays the role of $\chi_{0}$ in this lemma. The
constants $\delta_{n,i}$ are taken such that 
\[
\delta_{n,i}\defeq\begin{cases}
i\e^{-\nu T_{n}} & \mbox{if }1\leq i\leq m,\\
-(\mu_{n}-i+1)\e^{-\nu T_{n}} & \mbox{if }m+1\leq i\leq\mu_{n}.
\end{cases}
\]
In this way, the geodesics $(\gamma_{i})_{1\leq i\leq\mu_{n}}$ have
their lengths dilated (for $1\leq i\leq m$) or contracted (for $m+1\leq i\leq\mu_{n}$)
to $\ti\ell_{1},\dots,\ti\ell_{\mu_{n}}$ with 
\begin{align*}
\ti\ell_{i} & =\ell_{i}+i\e^{-\nu T_{n}}\qquad1\leq i\leq m,\\
\ti\ell_{j} & =\ell_{j}-(\mu_{n}-j+1)\e^{-\nu T_{n}},\qquad m+1\leq j\leq\mu_{n}.
\end{align*}
Define now 
\[
g_{n}=\prod_{i=1}^{\mu_{n}}(1+\frac{\delta_{n,i}}{\epsilon_{n}}\chi_{n,i})^{2}g_{n-1}=\left(\prod_{p=1}^{n}\prod_{i=1}^{\mu_{p}}\e^{2\log(1+\frac{\delta_{p,i}}{\epsilon_{p}}\chi_{p,i})}\right)g_{0}\defeq\e^{2F_{n}}g_{0}.
\]
To simplify the notations below, we write $f_{n,i}\defeq\delta_{n,i}\epsilon_{n}^{-1}\chi_{n,i}$.
By $C_{k}$, we will denote a positive constant depending only on
$g_{0},\eps_{0},\eps,k$ whose value may change from line to line.
Equation \eqref{eq: Margu grossier} gives $\mu_{n}\leq C_{0}\e^{hT_{n}}/T_{n}$,
so we have 
\[
\left|\frac{\delta_{n,i}}{\epsilon_{n}}\right|\leq\frac{2C_{0}\e^{\alpha}}{T_{n}}\e^{(h+\alpha-\nu)T_{n}}.
\]
We deduce from this equation and Lemma \ref{lem: conformal dilation}
that $f_{n,i}$ satisfies 
\begin{equation}
\|f_{n,i}\|_{C^{k}}\leq C_{k}\frac{\e^{-(\nu-h-(k+1)\alpha)T_{n}}}{T_{n}},\quad k\geq0.\label{eq: estimate f}
\end{equation}
For $k\geq1$, we have

\[
\frac{d^{k}}{dx^{k}}\log(1+f)=\frac{P_{r}(f^{(k)},\dots,f)}{(1+f)^{2^{k-1}}},
\]
where $P_{k}$ is a polynomial of degree $2^{k-1}$. Using \eqref{eq: estimate f}
and $T_{0}$ large enough, we then get 
\[
\|P_{k}(f_{n,i}^{(k)},\dots,f_{n,i})\|_{C^{0}}\leq C{}_{k}\frac{\e^{-(\nu-h-(k+1)\alpha)T_{n}}}{T_{n}}.
\]
Remark that for a given $n$, Corollary \ref{cor: boules separees}
ensures that $\supp f_{n,i}\cap\supp f_{n,j}=\emptyset$ if $j\neq i$.
This yields to 
\begin{equation}
\|\sum_{i=1}^{\mu_{n}}\log(1+f_{n,i})\|_{C^{k}}\leq C_{k}\frac{\e^{-(\nu-h-(k+1)\alpha)T_{n}}}{T_{n}},\quad k\geq0.\label{eq: estimate one step}
\end{equation}
If the constant $\nu$ satisfies 
\begin{equation}
\nu=h+(k+1)\alpha+\eps/2=(k+2)h+2(k+1)\kappa+\eps,\label{eq: condition nu}
\end{equation}
we have finally 
\[
\|F_{n}\|_{C^{k}}=\|\sum_{p=0}^{n}\sum_{i=1}^{\mu_{n}}\log(1+f_{n,i})\|_{C^{k}}\leq C_{k}\sum_{p=0}^{\infty}\frac{\e^{-(\nu-h-(k+1)\alpha)T_{p}}}{T_{p}}\leq C_{k}\frac{\e^{-(\nu-h-(k+1)\alpha)T_{0}}}{T_{0}}.
\]
From the above remarks, we end up with 

\begin{eqnarray}
\|(\e^{2F_{n}}-1)g_{0}\|_{C^{k}} & \leq & C(g_{0},\eps_{0},\eps,k)\frac{\e^{-(\nu-h-(k+1)\alpha)T_{0}}}{T_{0}}\,,\quad k\geq2.\label{eq: choice T0}
\end{eqnarray}
In particular, if $T_{0}=T_{0}(g_{0},\eps_{0},\eps,k)$ is large enough,
the right hand side of the previous equation is $<\eps_{0}$ and we
have $g_{n}\in\cM_{k}(\eps_{0})$ for all $n\in\N$.

By construction, the metric perturbation at step $n$ avoids any $\epsilon_{n}$-neighborhood
of $\scC_{g_{n-1}}(T_{n-1})$ and $\scC_{g_{n-1}}(]T_{n},T_{n+1}])$
: as a result, closed $g_{n-1}$-geodesics with length in $[T_{0},T_{n-1}]\cup]T_{n},T_{n+1}]$
remain unchanged for $g_{n}$. This means that : 
\begin{equation}
\scC_{g_{n-1}}(T_{n-1})\subset\scC_{g_{n}}(T_{n-1})\quad\text{and}\quad\scC_{g_{n-1}}(]T_{n},T_{n+1}])\subset\scC_{g_{n}}(]T_{n},T_{n+1}]).\label{eq: inclusions LSP}
\end{equation}
To have equalities, it remains to check that closed $g_{n-1}$-geodesics
with length $>T_{n+1}$ are not mapped by $f_{g_{n-1}\to g_{n}}$
to $g_{n}$-closed geodesics with length in $[0,T_{n}]$ \textendash{}
recall equation \eqref{eq: bijection of geodesics}. To see this,
let $\gamma\in\scC_{g_{n-1}}(]T_{n+1},+\infty[)$, and write $\ti\gamma=f_{g_{n-1}\to g_{n}}(\gamma)$.
From \eqref{eq: length perturb}, we have 
\begin{equation}
\frac{T_{n+1}}{\sqrt{1+\Delta_{n}}}\leq\ell_{g_{n}}(\ti\gamma)\leq\ell_{g_{n-1}}(\gamma)\sqrt{1+\Delta_{n}},\qquad\Delta_{n}\defeq\|g_{n}-g_{n-1}\|_{C^{0}}.\label{eq: higher lengths}
\end{equation}
On the other hand, equations \eqref{eq: estimate one step} and \eqref{eq: choice T0}
give 
\[
0\leq\Delta_{n}\leq\frac{2C_{0}\e^{\alpha}}{T_{n}}\e^{(h+\alpha-\nu)T_{n}}\|g_{n-1}\|_{C^{0}}\leq\frac{\e^{-(\nu-\alpha-h)T_{n}}}{T_{n}}C(g_{0},\eps_{0},\eps,k).
\]
In particular, since $k\geq2$, \eqref{eq: condition nu} gives $\nu\geq h+3\alpha$.
In view of the fact that $\sqrt{1+\Delta_{n}}\leq1+C\Delta_{n}$ for
some $C>0$, if $T_{0}$ is sufficiently large we will get from the
previous equation $C\Delta_{n}T_{n}<1$, and then 
\begin{eqnarray*}
\frac{T_{n+1}}{\sqrt{1+\Delta_{n}}} & \geq & \frac{T_{n+1}}{1+C\Delta_{n}}>T_{n}.
\end{eqnarray*}
This means that $\ell_{g_{n}}(\ti\gamma)>T_{n}$ and we can conclude
that $\scL_{g_{n}}\cap]T_{n-1},T_{n}]$ is precisely $\ti\ell_{1},\dots\ti\ell_{\mu_{n}}$.

It is easily verified that $\scC_{g_{n}}(]T_{n-1},T_{n}])$ is $\nu$-separated
: we have indeed $\ti\ell_{i+1}-\ti\ell_{i}=\ell_{i+1}-\ell_{i}+\e^{-\nu T_{n}}$
for $i\leq m-1$ or $i\geq m+1$. Using \eqref{eq: Margu grossier}
and \eqref{eq: condition nu} we also get 
\[
\ti\ell_{m+1}-\ti\ell_{m}\geq C_{0}^{-1}\e^{-hT_{n}}-2\mu_{n}\e^{-\nu T_{n}}\geq\frac{1}{2}C_{0}^{-1}\e^{-h_{}T_{n}}>\e^{-\nu T_{n}}
\]
for $T_{0}$ large enough. As noted above, the closed $g_{n-1}$-geodesics
with length $\leq T_{n-1}$ have not been modified, so $\scC_{g_{n}}([T_{0},T_{n}])$
is $\nu$-separated if $\scC_{g_{n-1}}([T_{0},T_{n-1}])$ was.

It is now straightforward to check that we can start the above process
at some time $T_{0}=T_{0}(g_{0},\eps_{0},\eps,k)$ and get a sequence
of metrics $(g_{n})_{n\in\N}\in\cM_{k}(\eps_{0})$ with 
\[
g_{n}\xrightarrow[\|\cdot\|_{C^{k}}]{n\to\infty}g_{\infty},\qquad g_{\infty}\defeq\prod_{n=1}^{\infty}\prod_{i=1}^{\mu_{n}}(1+\frac{\delta_{n,i}}{\epsilon_{n}}\chi_{n,i})^{2}g_{0},\qquad\|g_{0}-g_{\infty}\|_{C^{k}}<\eps_{0}.
\]
By construction, $\scC_{g_{\infty}}([T_{0},+\infty[)$ is $\nu$-separated.
The conclusion of Theorem \ref{thm: split} now follows readily.

\medskip{}

\noun{Acknowledgements : }The author thanks Pierre Bergé and Giona
Veronelli for stimulating discussions. This work has been partially
supported by the Agence Nationale de la Recherche, under the grant
Gerasic-ANR-13-BS01-0007-0.

\bibliographystyle{amsalpha}
\bibliography{ESbiblio_LSG}

\end{document}